\documentclass[11pt]{amsart}

\usepackage[english]{babel}
\usepackage{amsmath}
\usepackage{amsfonts}
\usepackage{amssymb}
\usepackage{amsthm}
\usepackage{epsfig}
\usepackage{graphicx}
\usepackage{url,hyperref}

\newtheorem{theorem}{Theorem}[section]
\newtheorem{lemma}{Lemma}[section]
\newtheorem{corollary}{Corollary}[section]
\newtheorem{proposition}{Proposition}[section]

\newtheorem{remark}{Remark}[section]

\newcounter{theor}

\DeclareMathOperator{\inter}{int}
\def\pos{\mathop\mathrm{pos}\nolimits}
\def\span{\mathop\mathrm{span}\nolimits}

\def\conv{\mathop\mathrm{conv}\nolimits}
\def\bd{\mathop\mathrm{bd}\nolimits}

\def\s{\mathbb{S}}

\def\S{\mathcal{S}}

\def\rd{\widetilde{\mathcal{R}}}
\def\R{\mathbb{R}}
\def\N{\mathbb{N}}

\def\V{\mathrm{V}}

\def\C{\mathbb{C}}

\def\W{\mathrm{W}}
\def\Wd{\widetilde{\mathrm{W}}}

\def\Wdk#1#2{\widetilde{\mathrm{W}}^{^{(#1)}}_{#2}\!}

\def\e{\mathrm{e}}
\def\im{\mathrm{i}}
\def\sy#1#2{\sigma_{#1}\left(#2\right)}

\def\Real{\mathop\mathrm{Re}\nolimits}
\def\Im{\mathop\mathrm{Im}\nolimits}

\def\f#1#2{\widetilde{f}_{#1,#2}}
\newcommand{\dlat}{\mathrm{d}}

\newcommand{\bigdotcup}[1][0pt]{\mathaccent\cdot{}\kern-#1\bigcup}

\numberwithin{equation}{section}

\begin{document}

\title[A characterization of dual querma\ss integrals]%
{A characterization of dual querma\ss integrals and the roots of dual
Steiner polynomials}

\author[D. Alonso-Guti\'errez]{David Alonso-Guti\'errez}
\address{Departamento de Matem\'aticas, Universidad de Zaragoza, C/ Pedro Cerbuna 12,
50009-Zaragoza, Spain} \email{alonsod@unizar.es}

\author[M. Henk]{Martin Henk}
\address{Institut f\"ur Mathematik, Technische Universit\"at Berlin, Stra\ss e des 17. Juni
136, D-10623 Berlin, Germany} \email{henk@math.tu-berlin.de}

\author[M. A. Hern\'andez Cifre]{Mar\'\i a A. Hern\'andez Cifre}
\address{Departamento de Matem\'aticas, Universidad de Murcia, Campus de
Espinar\-do, 30100-Murcia, Spain} \email{mhcifre@um.es}

\thanks{Supported by: MINECO/FEDER projects MTM2016-77710-P and
MTM2015-65430-P; ``Programa de Ayudas a Grupos de Excelencia de la
Regi\'on de Murcia'', Fundaci\'on S\'eneca, 19901/GERM/15.}



\begin{abstract}
For any $I\subset\R$ finite with $0\in I$, we provide a characterization
of those tuples $(\omega_i)_{i\in I}$ of positive numbers which are dual
querma\ss integrals of two star bodies. It turns out that this problem is
related to the moment problem. Based on this relation we also get new
inequalities for the dual querma\ss integrals. Moreover, the above
characterization will be the key tool in order to investigate structural
properties of the set of roots of dual Steiner polynomials of star bodies.
\end{abstract}

\maketitle

\section{Introduction and notation}

A subset $S$ of the $n$-dimensional Euclidean space $\R^n$ is called {\em
starshaped} (with respect to the origin $0$) if $S\neq\emptyset$ and the
segment $[0,x]\subseteq S$ for all $x\in S$. For a compact starshaped set
$K$ its {\em radial function} $\rho_K$ is defined by
\[
\rho_K(u)=\max\{\lambda\geq 0:\lambda u\in K\}, \quad u\in\s^{n-1}.
\]
Moreover, a {\em star body} is a compact starshaped set with positive and
continuous radial function. We observe that this property implies that any
star body has non-empty interior. We will denote by $\S^n_0$ the set of
all star bodies in $\R^n$. In particular, convex bodies (compact and
convex sets) containing the origin in its interior are star bodies, and we
write $B^n_2$ to denote the $n$-dimensional unit ball. The volume of a set
$M\subset\R^n$, i.e., its $n$-dimensional Lebesgue measure, is denoted by
$|M|$, or $|M|_n$ if the distinction of the dimension is needed.
Furthermore, we write $\bd M$ and $\inter M$ to represent the boundary and
the interior of $M$, and we use $\conv M$ and $\pos M$ for its convex and
positive hulls, respectively.

For two convex bodies $K,L$ and a non-negative real number $\lambda$, the
volume of the Minkowski sum (vectorial addition) $K+\lambda\,L$ is
expressed as a polynomial of degree at most $n$ in $\lambda$ (see
\cite{Ste40}), and it is written as
\begin{equation}\label{eq:steiner-minkowski}
|K+\lambda L|=\sum_{i=0}^n\binom{n}{i}\W_i(K,L)\lambda^i.
\end{equation}
This expression is called {\em relative Steiner formula} of $K$, and the
coefficients $\W_i(K,L)$ are the {\em relative querma\ss integrals} of
$K$, special cases of the more generally defined {\em mixed volumes} (see
e.g. \cite[s.~5.1]{Sch}).

{\em Dual Brunn-Minkowski theory} goes back to Lutwak \cite{Lut1,Lut2},
and it is a cornerstone of modern convex geometry. For the immense impact
of this theory we refer e.g. to \cite{BHP, HLYZ, Sch} and the references
inside. In this context, and among others, convex bodies are replaced by
star bodies, the Minkowski sum by the radial addition and the support
function by the radial function. For $x,y\in\R^n$, the radial addition
$x\widetilde{+}y$ is defined as
\[
x\widetilde{+}y=\left\{\begin{array}{ll}
 x+y & \;\text{ if $x,y$ are linearly dependent,}\\
 0 & \;\text{ otherwise.}
\end{array}\right.
\]
Then, the {\em radial sum} $K\widetilde{+}L$ for $K,L\in\S^n_0$ is defined
by
\[
K\widetilde{+}L=\left\{x\widetilde{+}y:x\in K,\,y\in L\right\},
\]
and has the property that, for $\lambda,\mu\geq 0$,
\[
\rho_{\mu K\widetilde{+}\lambda L}=\mu\rho_K+\lambda\rho_L.
\]
As in the classical case, the volume of the radial sum
$K\widetilde{+}\lambda L$ is also expressed as a polynomial of degree $n$
in $\lambda$ (see e.g. \cite[p.~508]{Sch}),
\begin{equation}\label{eq:dual_steiner}
\bigl|K\widetilde{+}\lambda L\bigr|
=\sum_{i=0}^n\binom{n}{i}\Wd_i(K,L)\lambda^i.
\end{equation}
This expression is known as {\em dual Steiner formula} of $K$. The
coefficients $\Wd_i(K,L)$ are the {\em dual querma\ss integrals} of $K$
and $L$, and can be expressed in terms of their radial functions as
\begin{equation}\label{eq:dual_quermass}
\Wd_i(K,L)=\dfrac{1}{n}\int_{\s^{n-1}}\rho_K(u)^{n-i}\rho_L(u)^i\,\dlat\sigma(u).
\end{equation}
Here $\sigma$ is the usual spherical Lebesgue measure. In particular, the
use of spherical coordinates immediately yields $\Wd_0(K,L)=|K|$,
$\Wd_n(K,L)=|L|$ and $2\Wd_{n-1}(K,B_2^n)/|B_2^n|$ is the average length
of chords of $K$ through the origin. Moreover, and in contrast to the
classical querma\ss integrals, the dual ones can be defined via
\eqref{eq:dual_quermass} for any real index $i\in\R$. Dual querma\ss
integrals are particular cases of the dual mixed volumes, which were
introduced for the first time by Lutwak in \cite{Lut1} (see also
\cite[s.~9.3]{Sch}).

In \cite{HHCS} it was studied the problem whether Steiner polynomials can
be characterized, i.e.: given a polynomial $f(z)=\sum_{i=0}^na_iz^i$,
$a_i\geq 0$, is it a Steiner polynomial for a pair of convex bodies? This
question is equivalent to know when a set of $n+1$ non-negative real
numbers $\W_0,\dots,\W_n\geq 0$ arises as the set of relative querma\ss
integrals of two convex bodies. Shephard proved in \cite{She} that it
suffices that the numbers satisfy the well-known Aleksandrov-Fenchel
inequalities (see e.g. \cite[(9.40)]{Sch}) in order to be querma\ss
integrals of convex bodies (see also \cite{HHCS}).

This question is a particular case of a beautiful problem which is still
open: Given $r\geq 2$ convex bodies $K_1,\dots,K_r\subset\R^n$, there are
$N=\binom{n+r-1}{n}$ mixed volumes $\V(K_{i_1},\dots,K_{i_n})$, $1\leq
i_1\leq\dots\leq i_n\leq r$; for the definition and a deep study on mixed
volumes we refer to \cite[s.~5.1]{Sch}. Then, a set of inequalities is
said to be a {\it full set} if given $N$ (non-negative) numbers satisfying
the inequalities, they arise as the mixed volumes of $r$ convex bodies.
For $n=2$ and $r=3$, Heine \cite{Hei} proved that the Aleksandrov-Fenchel
inequalities together with the determinantal inequality
$\det\bigl(\V(K_i,K_j)_{i,j=1}^3\bigr)\geq 0$ are a full set. For $n\geq
2$, Shephard \cite{She} investigated whether the known inequalities
(Aleksandrov-Fenchel and some determinantal inequalities) are a full set,
and solved it for $r=2$. Moreover, he showed that for $r=n+2$ they do not
form a full set. For arbitrary $r$, the problem is still open.

The main aim of this paper is to consider the corresponding question in
the dual setting, i.e., to look for necessary and sufficient conditions
for $n+1$ positive real numbers to be the dual querma\ss integrals of two
star bodies in $\R^n$. The first substantial difference with the classical
case is that dual querma\ss integrals can be defined for any real index.
Hence we provide, for any finite subset $I\subset\R$ with $0\in I$, a
characterization of those tuples of positive numbers $(\omega_i)_{i\in I}$
which are the dual querma\ss integrals of two {\it $n$-dimensional} star
bodies.

In order to state our result we need the following notation: for $0<a<b$
and for an index set $I=\{0,i_1,\dots,i_m\}\subset\R$ of cardinality
$\#I=m+1$, we write
\begin{equation}
C_{a,b}^I=\pos\bigl\{(1,t^{i_1},t^{i_2},\dots,t^{i_m}):t\in[a,b]\bigr\}\subset\R^{m+1}.
\end{equation}
From now on and for the sake of brevity, any index set
$I=\{0,i_1,\dots,i_m\}\subset\R$ will be assumed to have cardinality
$\#I=m+1$.

\begin{theorem}\label{CharacterizationDualMixedVolumes1}
Let $I=\{0,i_1,\dots,i_m\}\subset\R$, let $(\omega_i)_{i\in I}$ be a
sequence of $m+1$ positive numbers and let $n\geq 2$. Then there exist
star bodies $K,L\in\S^n_0$ such that
\[
\Wd_i(K,L)=\omega_i,\quad\text{ for all }\;i\in I,
\]
if and only if either there exist $0<a<b$ such that
\[
(\omega_0,\omega_{i_1},\dots,\omega_{i_m})\in\inter C_{a,b}^I,
\]
or $\omega_i=\lambda^i\omega_0$ for some $\lambda>0$ and every $i\in I$;
in this case $L=\lambda K$.
\end{theorem}

Our above characterization of dual querma\ss integrals is related to the
\emph{moment problem} (see e.g. \cite{KrNu}). In Section
\ref{s:dual_quermass} we study this relation, which will allow us to get,
in the particular case when $I=\{0,1,\dots,m\}\subset\N$, new inequalities
between dual querma\ss integrals as a direct consequence of the following
more general result.

\begin{theorem}\label{t:family_ineq_dual}
Let $K,L\in\S^n_0$ be two star bodies. Then, for every $m\in\N$ the Hankel
matrices
\[
A_m=\Bigl(\Wd_{i+j}(K,L)\Bigr)_{i,j=0}^m,\quad
B_m=\Bigl(\Wd_{i+j+1}(K,L)\Bigr)_{i,j=0}^{m-1}
\]
are positive definite.
\end{theorem}

From the above theorem new {\it determinantal inequalities} for dual
querma\ss integrals are obtained.

\begin{corollary}\label{c:ineq_dual}
Let $K,L\in\S^n_0$ be two star bodies. We write $\Wd_i=\Wd_i(K,L)$ and
let, for every $m\in\N$,
\[
\Delta_m=\begin{pmatrix}
 \Wd_0 & \Wd_1 & \!\cdots\! & \Wd_m\\
\Wd_1 & \Wd_2 & \!\cdots\! & \Wd_{m+1}\\
\vdots & \vdots & \vdots & \vdots\\
\Wd_m & \Wd_{m+1} & \!\cdots\! & \Wd_{2m}
\end{pmatrix},\;\,
\Delta_m'=\begin{pmatrix}
 \Wd_1 & \Wd_2 & \!\cdots\! & \Wd_m\\
\Wd_2 & \Wd_3 & \!\cdots\! & \Wd_{m+1}\\
\vdots & \vdots & \vdots & \vdots\\
\Wd_m & \Wd_{m+1} & \!\cdots\! & \Wd_{2m-1}
\end{pmatrix}.
\]
Then we have the determinantal inequalities
\[
\det\Delta_m>0\quad\text{ and }\quad \det\Delta_m'>0.
\]
\end{corollary}

In the classical setting, the validity of a family of determinantal
inequalities for mixed volumes remains an open problem (see~\cite{BZ}).

\smallskip

In several recent articles (see e.g. \cite{HHCS} and the references
therein), the characterization of the querma\ss integrals of convex bodies
became a key tool in order to study properties of the roots of the
relative Steiner polynomial (regarded as a polynomial in a complex
variable, cf. \eqref{eq:steiner-minkowski}): structural properties of the
set of roots, convexity, closeness, monotonicity, stability, etc.

In this paper we also carry out the corresponding study for the roots of
{\em dual Steiner polynomials}. In the following we regard the right hand
side in \eqref{eq:dual_steiner} as a formal polynomial in a complex
variable $z\in\C$, which we denote~by
\[
\f{K}{L}(z)=\sum_{i=0}^n\binom{n}{i}\Wd_i(K,L)z^i.
\]
We observe that $0$ cannot be a root of any dual Steiner polynomial
because $\inter K\neq\emptyset$ for all $K\in\S^n_0$. Moreover, since
$\Wd_i(K,L)=\Wd_{n-i}(L,K)$ (see \eqref{eq:dual_quermass}), we have
$\f{K}{L}(z)=z^n\,\f{L}{K}(1/z)$, and thus, up to multiplication by real
constants, $\f{K}{L}(z)$ and $\f{L}{K}(z)$ have the same roots.

Here we are interested in the location and the structure of the roots of
$\f{K}{L}(z)$. To this end, let $\C^+=\{z\in\C:\Im(z)\geq 0\}$, and we
denote by $\R_{<0}$ and $\R_{\geq 0}$ the negative and non-negative real
axes, respectively. For any dimension $n\geq 2$, let
\begin{equation}\label{e:R_n}
\rd(n)=\bigl\{z\in\C^+: \f{K}{L}(z)=0 \text{ for some }K,L\in\S^n_o\bigr\}
\end{equation}
be the set of all roots of all dual Steiner polynomials in the upper
half-plane. We prove the following result.

\begin{theorem}\label{t:cone_half-open_monotone}
The set of roots $\rd(n)$ satisfies the following properties:
\begin{enumerate}\itemsep0pt
\item[(a)] It is a convex cone, containing the negative real axis.
\item[(b)] It is half-open, i.e., it does not contain the ray of the
boundary not consisting of $\R_{<0}$.
\item[(c)] It is monotonous in the dimension, i.e.,
$\rd(n)\subseteq\rd(n+1)$.
\end{enumerate}
\end{theorem}

We observe that the dual Steiner polynomial shares properties (a) and (c)
with the relative Steiner polynomial (see \cite[Theorem~1.1 and
Theorem~1.3]{HHCS}). However, property (b) provides a first structural
difference between both polynomials, since the cone of roots of the
classical Steiner polynomial is shown to be closed (see
\cite[Theorem~1.2]{HHCS}).

The above theorem will be proved in Section \ref{s:roots}, along with
several additional properties of the roots. In Section
\ref{s:dual_quermass} we give the proof of
Theorems~\ref{CharacterizationDualMixedVolumes1} and
\ref{t:family_ineq_dual}, which are based on a relation with the moment
problem.

\section{Dual querma\ss integrals and the moment problem}\label{s:dual_quermass}

For any $i\in\R$, the $i$-th dual querma\ss integral is a monotonous and
homogeneous functional of degree $n-i$ in its first argument and of degree
$i$ in the second one (cf. \eqref{eq:dual_quermass}), i.e.: given
$K,K',L\in\S^n_0$ with $K\subseteq K'$ and $\lambda>0$, then
\[
\begin{split}
\Wd_i(K,L) & \leq\Wd_i(K',L)\quad\text{ and}\\
\Wd_i(\lambda K,L) & =\lambda^{n-i}\Wd_i(K,L),\quad
    \Wd_i(K,\lambda L)=\lambda^i\Wd_i(K,L),
\end{split}
\]
for any $i\in\R$. It is also well-known that the dual querma\ss integrals
of two star bodies $K,L$ satisfy the inequalities
\begin{equation}\label{e:special_dual_af}
\Wd_j(K,L)^{k-i}\leq\Wd_i(K,L)^{k-j}\Wd_k(K,L)^{j-i},\quad\; i<j<k,
\end{equation}
the counterpart to the classical Aleksandrov-Fenchel inequalities (see
e.g. \cite[(9.40)]{Sch}), but now $i,j,k\in\R$ is allowed. In
\eqref{e:special_dual_af} equality holds if and only if $K$ and $L$ are
dilates.

We start this section collecting some further easy properties that will be
needed later on.
\begin{lemma}\label{l:quermass_prop}
Let $K,L\in\S^n_0$ and let $i,j,k\in\R$, with $i<j<k$.
\begin{enumerate}\itemsep0pt
\item[i)] If $L\subseteq K$ then
\begin{equation}\label{e:Favard_dual}
\Wd_i(K,L)\geq\Wd_j(K,L).
\end{equation}
\item[ii)] $\Wd_l(K,L)=\mu^l\Wd_0(K,L)$ for $l=i,j,k$ and some
$\mu>0$ if and only if $L=\mu K$.
\end{enumerate}
\end{lemma}

\begin{proof}
i) is a direct consequence of the monotonicity. In order to prove ii) we
observe that if $\Wd_l(K,L)=\mu^l\Wd_0(K,L)$ for $l=i,j,k$, then we get
\[
\Wd_i(K,L)^{k-j}\Wd_k(K,L)^{j-i}=\mu^{j(k-i)}\Wd_0(K,L)^{k-i}=\Wd_j(K,L)^{k-i}.
\]
Hence, we have equality in the dual Aleksandrov-Fenchel inequality
\eqref{e:special_dual_af}, which yields $L=\mu K$. The converse is
obvious.
\end{proof}

\subsection{Characterizing dual querma\ss integrals: proof of Theorem~\ref{CharacterizationDualMixedVolumes1}}

The main purpose of this section is to prove
Theorem~\ref{CharacterizationDualMixedVolumes1}. To this end we first show
a couple of lemmas, for which we need the following notation: given a
measure $\mu$ on an interval $[a,b]\subset(0,\infty)$ we write, for any
$i\in\R$,
\[
m_i(\mu)=\int_a^bt^i\,\dlat\mu(t).
\]
We observe that when $i\in\N\cup\{0\}$, the above numbers are the moments
of $\mu$ on the interval $[a,b]$ (see Subsection~\ref{ss:moment} for a
brief introduction to the moment problem).

\begin{lemma}\label{l:lemma1}
Let $[a,b]\subset(0,\infty)$. For $I=\{0,i_1,\dots,i_m\}\subset\R$, let
$(\omega_i)_{i\in I}$ be a sequence of $m+1$ positive numbers with
$\omega_0=|B_2^n|$. Let $\mu$ be a positive measure on $[a,b]$ such that
$\mu\bigl([c,d)\bigr)>0$ for every $[c,d)\subset[a,b]$ and
\[
m_i(\mu)=\omega_i \quad\text{ for every }\;i\in I.
\]
Then there exists $L\in\S^n_0$ satisfying $\omega_i=\Wd_i(B_2^n,L)$ for
all $i\in I$.
\end{lemma}

\begin{proof}
Let $F:[a,b]\longrightarrow[0,1]$ be the function defined by
\[
F(t)=\frac{\mu\bigl([t,b]\bigr)}{\mu\bigl([a,b]\bigr)}.
\]
Our assumption ensures that $F$ is a strictly decreasing function and
continuous from the left, and satisfies $F(a)=1$ and
$F(b)=\mu\bigl(\{b\}\bigr)/\mu\bigl([a,b]\bigr)$. Let
$G:[0,1]\longrightarrow[a,b]$ be the function
\[
G(s)=\sup\bigl\{t\in[a,b]:F(t)\geq s\bigr\},
\]
which coincides with $F^{-1}$ when $F$ is bijective. Since $F$ is strictly
decreasing, it is easy to see that $G$ is decreasing and continuous, and
so, the function $\rho_L:\s^{n-1}\longrightarrow[a,b]$ given by
\[
\rho_L(u)=G\left(\frac{\sigma\Bigl(\bigl\{v\in\s^{n-1}:|v_1|\geq|u_1|\bigr\}\Bigr)}{\sigma(\s^{n-1})}\right)
\]
is continuous on $\s^{n-1}$ and hence defines a star body $L$. Moreover,
it is clear that $\rho_L(v)\geq\rho_L(u)$ if and only if $|v_1|\geq
|u_1|$. Therefore, if $t\in[a,b]$ then
\[
\frac{\sigma\Bigl(\bigl\{v\in\s^{n-1}:\rho_L(v)\geq
t\bigr\}\Bigr)}{\sigma(\s^{n-1})}=F(t)=\frac{\mu\bigl([t,b]\bigr)}{\mu\bigl([a,b]\bigr)},
\]
whereas for $t\notin[a,b]$ we trivially have
\[
\frac{\sigma\Bigl(\bigl\{v\in\s^{n-1}:\rho_L(v)\geq
t\bigr\}\Bigr)}{\sigma(\s^{n-1})}=\frac{\mu\Bigl(\bigl\{s\in[a,b]:s\geq
t\bigr\}\Bigr)}{\mu\bigl([a,b]\bigr)}.
\]
Finally, since $\mu\bigl([a,b]\bigr)=m_0(\mu)=\omega_0=|B_2^n|$ and
$\sigma(\s^{n-1})=n|B^n_2|$, we get, for every $i\in I$,
\[
\begin{split}
\Wd_i(B_2^n,L) &
    =\frac{1}{n}\int_{\s^{n-1}}\rho_L^i(u)\,\dlat\sigma(u)\\
 & =\frac{1}{n}\int_0^\infty it^{i-1}\sigma\Bigl(\bigl\{v\in\s^{n-1}:\rho_L(v)\geq t\bigr\}\Bigr)\,\dlat t\\
 & =\frac{|B_2^n|}{\mu\bigl([a,b]\bigr)}\int_0^\infty it^{i-1}\mu\Bigl(\bigl\{s\in[a,b]:s\geq t\bigr\}\Bigr)\,\dlat t\\
 & =\int_a^bs^i\,\dlat\mu(s)=m_i(\mu)=\omega_i.\qedhere
\end{split}
\]
\end{proof}

A refinement of Riesz's Theorem will be also a key tool in our proof (see
e.g. \cite[Theorem~3.5 and P.~3.9 in p.~17]{KrNu}). It provides the
connection between our characterization of dual querma\ss integrals and
the moment problem, for which we refer to Subsection \ref{ss:moment}.

\begin{theorem}[Riesz]\label{t:Riesz}
Let $\alpha:[a,b]\longrightarrow\R^n$,
$\alpha(t)=\bigl(\alpha_1(t),\dots,\alpha_n(t)\bigr)$, be a curve in
$\R^n$ and let $x=(x_1,\dots,x_n)\in\R^n$. There exists a probability
measure $\mu$ on $[a,b]$ such that
\[
x_i=\int_a^b\alpha_i(t)\,\dlat\mu(t),\quad\text{ for every } i=1,\dots n,
\]
if and only if $x\in\conv\bigl\{\alpha(t):t\in[a,b]\bigr\}$.

Moreover, $x\in\inter\conv\bigl\{\alpha(t):t\in[a,b]\bigr\}$ if and only
if there exits a continuous function $\phi:[a,b]\longrightarrow(0,\infty)$
such that $\dlat\mu(t)=\phi(t)\,\dlat t$.
\end{theorem}

The following lemma shows that the above property is also equivalent to
the fact that the measure $\mu$ can be assumed to be supported on the
whole interval $[a,b]$.

\begin{lemma}\label{lemmaMeasuresFullSupport}
Let $\alpha:[a,b]\longrightarrow\R^n$ be a continuous curve in $\R^n$ not
contained in a hyperplane and let $x=(x_1,\dots,x_n)\in\R^n$. There exists
a probability measure $\mu$ on $[a,b]$ such that $\mu\bigl([c,d)\bigr)>0$
for every $[c,d)\subset[a,b]$ and
\begin{equation}\label{e:eq_lemmaMeasuresFullSupport}
x_i=\int_a^b\alpha_i(t)\,\dlat\mu(t),\quad\text{ for every } i=1,\dots n,
\end{equation}
if and only if $x\in\inter\conv\bigl\{\alpha(t):t\in[a,b]\bigr\}$.
\end{lemma}

\begin{proof}
First we suppose that $x\in\inter\conv\bigl\{\alpha(t):t\in[a,b]\bigr\}$.
Then, by Theorem \ref{t:Riesz}, there exists a probability measure $\mu$
with a positive density $\phi$ with respect to the Lebesgue measure
satisfying \eqref{e:eq_lemmaMeasuresFullSupport}, and thus
$\mu\bigl([c,d)\bigr)>0$ for all $[c,d)\subset[a,b]$.

Conversely, if we suppose the existence of a measure $\mu$ satisfying our
hypotheses, Theorem~\ref{t:Riesz} ensures that
$x\in\conv\bigl\{\alpha(t):t\in[a,b]\bigr\}$. So, let us assume that
$x\in\bd\conv\bigl\{\alpha(t):t\in[a,b]\bigr\}$. Then there exists a
supporting hyperplane to $\conv\bigl\{\alpha(t):t\in[a,b]\bigr\}$ at $x$
with outer normal vector $u\in\s^{n-1}$ such that $\langle y,u\rangle\leq
\langle x,u\rangle$ for every
$y\in\conv\bigl\{\alpha(t):t\in[a,b]\bigr\}$. Furthermore, since
$\alpha\bigl([a,b]\bigr)$ is not contained in a hyperplane, there exists
$[c,d)\subset[a,b]$ such that $\langle\alpha(t),u\rangle<\langle
x,u\rangle$ for every $t\in[c,d)$, and thus
\[
\langle x,u\rangle
=\int_a^b\bigl\langle\alpha(t),u\bigr\rangle\,\dlat\mu(t)<\langle
x,u\rangle,
\]
a contradiction. Therefore,
$x\in\inter\conv\bigl\{\alpha(t):t\in[a,b]\bigr\}$.
\end{proof}

Now we are ready to prove Theorem \ref{CharacterizationDualMixedVolumes1},
for which we will apply the above results to the curve
$\alpha(t)=(t^{i_1},t^{i_2},\dots,t^{i_m})$. This is called the {\it
moment curve} when $i_k=k$ for $k=1,\dots,m$.

\begin{proof}[Proof of Theorem \ref{CharacterizationDualMixedVolumes1}]
We start assuming the existence of star bodies $K,L\in\S^n_0$ such that
$\omega_i=\Wd_i(K,L)$ for all $i\in I$.

First we observe that
\begin{equation}\label{e:dualW_i_int2}
\Wd_i(K,L)=\frac{1}{n}\int_{\s^{n-1}}\left(\frac{\rho_L(u)}{\rho_K(u)}\right)^i\rho_K^{n}(u)\,\dlat\sigma(u).
\end{equation}
We denote by $f:\s^{n-1}\longrightarrow(0,\infty)$ the continuous function
given by
\[
f(u)=\frac{\rho_L(u)}{\rho_K(u)},
\]
and let $a=\min_{u\in\s^{n-1}}f(u)$ and $b=\max_{u\in\s^{n-1}}f(u)$.

Let $\nu$ be the measure on the sphere given by
$\dlat\nu=(1/n)\rho_K^n(u)\,\dlat\sigma(u)$ and let $\mu$ be the
push-forward measure of $\nu$ by $f$. Then $\mu$ is supported on $[a,b]$
and is defined by
\[
\mu(A)=\nu\bigl(f^{-1}(A)\bigr)=\frac{1}{n}\int_{f^{-1}(A)}\rho_K^n(u)\,\dlat\sigma(u)
\]
for any Borel subset $A\subseteq [a,b]$. Consequently, if $a<b$ we have
\[
\begin{split}
\int_a^bt^i\,\dlat\mu(t) & =\int_{\s^{n-1}}f(u)^i\,\dlat\nu(u)
    =\frac{1}{n}\int_{\s^{n-1}}\left(\frac{\rho_L(u)}{\rho_K(u)}\right)^i\rho_K^{n}(u)\,\dlat\sigma(u)\\
 & =\Wd_i(K,L)=\omega_i
\end{split}
\]
for every $i\in I$, and $\dlat\mu/\Wd_0(K,L)$ is a probability measure
with support $[a,b]$. Then, the second part in Theorem \ref{t:Riesz}
ensures that
\[
\left(\frac{\omega_{i_1}}{\omega_0},\frac{\omega_{i_2}}{\omega_0},\dots,\frac{\omega_{i_m}}{\omega_0}\right)
\in\inter\conv\bigl\{(t^{i_1},t^{i_2},\dots,t^{i_m}):t\in[a,b]\bigr\}
\]
and hence
\[
(\omega_0,\omega_{i_1},\dots,\omega_{i_m})\in\inter C_{a,b}^I.
\]
Now, if $a=b$ then $f(u)=a$ for every $u\in\s^{n-1}$, which implies that
$L=aK$, and hence $\omega_i=\Wd_i(K,L)=a^i|K|=a^i\omega_0$ for all $i\in
I$.

For the converse, we first assume that there exist $0<a<b$ such that
\[
(\omega_0,\omega_{i_1},\dots,\omega_{i_m})\in\inter C_{a,b}^I.
\]
Then
\[
\left(\frac{\omega_{i_1}}{\omega_0},\frac{\omega_{i_2}}{\omega_0},\dots,\frac{\omega_{i_m}}{\omega_0}\right)
\in\inter\conv\bigl\{(t^{i_1},t^{i_2},\dots,t^{i_m}):t\in[a,b]\bigr\},
\]
and Lemma \ref{lemmaMeasuresFullSupport} ensures the existence of a
probability measure $\mu$ on $[a,b]$ such that $\mu\bigl([c,d)\bigr)>0$
for every $[c,d)\subset[a,b]$ and
\[
\frac{\omega_i}{\omega_0}=\int_a^bt^i\,\dlat\mu(t),\quad i\in I.
\]
If we write
\[
\omega_i'=\frac{|B_2^n|\omega_i}{\omega_0}=\int_a^bt^i|B_2^n|\,\dlat\mu(t)=m_i\bigl(|B_2^n|\mu\bigr),
\]
we can apply Lemma \ref{l:lemma1} to the set
$\left\{\omega_0',\omega_{i_1}'\dots,\omega_{i_m}'\right\}$ and obtain the
existence of a star body $L'\in\S^n_0$ such that
\[
\frac{|B_2^n|\omega_i}{\omega_0}=\omega_i'=\Wd_i(B_2^n,L')
\]
for $i\in I$. Then, just taking
\[
K=\left(\frac{\omega_0}{|B_2^n|}\right)^{1/n}B_2^n\quad\text{ and } \quad
L=\left(\frac{\omega_0}{|B_2^n|}\right)^{1/n}L',
\]
we get
\[
\omega_i=\Wd_i(K,L)\quad\text{ for all }\; i\in I.
\]
Finally, if there exists $\lambda>0$ such that
$\omega_i=\lambda^i\omega_0$ for all $i\in I$, any star body $K\in\S^n_0$
with $|K|=\omega_0$ and $L=\lambda K$ yield $\omega_i=\Wd_i(K,L)$.
\end{proof}

From now on, when $I=\{0,\dots,m\}\subset\N$, we will write $C_{a,b}^m$
for the cone $C_{a,b}^I$ in order to stress its dimension. At this point
we would like to notice the following fact, which has be used in the above
proof: a point
\[
(x_0,x_1,\dots,x_m)\in\inter C_{a,b}^m
\]
if and only if
\[
\left(\frac{x_1}{x_0},\dots,\frac{x_m}{x_0}\right)\in\inter\conv\bigl\{(t^1,t^2,\dots,t^m):t\in[a,b]\bigr\}.
\]
The set $\conv\bigl\{(t^1,t^2,\dots,t^m):t\in[a,b]\bigr\}$ is called the
{\it cyclic body} associated to $[a,b]$, and it is just the union of all
cyclic polytopes in $[a,b]$. So, in order to determine if a point lies or
not in its interior, it is convenient to know its facial structure and
supporting hyperplanes. This problem has been studied  and solved in
\cite{Pu}.

In the particular case when $I=\{0,1,\dots,m\}$, $m\in\N$, another
characterization of those sequences of positive numbers which are dual
querma\ss integrals of two star bodies can be obtained from
Theorem~\ref{CharacterizationDualMixedVolumes1} and the following fact
(see \cite[Theorem 1.1 in c.~3]{KrNu}):
\begin{equation}\label{t:polynom>0}
\mbox{\begin{minipage}{0.85\textwidth} Let $0<a<b$ and $m\in\N$. Then
$(x_0,\dots,x_m)\in\inter C_{a,b}^m$ if and only if, for every polynomial
$\sum_{i=0}^mc_it^i$ which is positive on $[a,b]$,
\[
\sum_{i=0}^mc_ix_i>0.
\]
\end{minipage}
}
\end{equation}

\begin{theorem}\label{CharacterizationDualMixedVolumes2}
Let $m,n\in\N$, $n\geq 2$, and let $(\omega_i)_{i=0}^m$ be a sequence of
$m+1$ positive numbers. Then there exist star bodies $K,L\in\S^n_0$ such
that
\[
\Wd_i(K,L)=\omega_i\quad\text{ for all }\;i=0,\dots m
\]
if and only if, either there exist $0<a<b$ such that the Hankel matrices
$(a_{j,k})_{j,k=0}^r$ and $(b_{j,k})_{j,k=0}^r$ given by
\[
\begin{split}
a_{j,k} & =\left\{\begin{array}{ll}
        \omega_{j+k} & \;\text{ if } m=2r,\\[1mm]
        \omega_{j+k+1}-a\omega_{j+k} & \;\text{ if } m=2r+1,
        \end{array}\right.\\[2mm]
b_{j,k} & =\left\{\begin{array}{ll}
        (a+b)\omega_{j+k+1}-ab\omega_{j+k}-\omega_{j+k+2} & \;\text{ if } m=2r,\\[1mm]
        b\omega_{j+k}-\omega_{j+k+1} & \;\text{ if } m=2r+1,
        \end{array}\right.
\end{split}
\]
are positive definite, or $\omega_i=\lambda^i\omega_0$ for some
$\lambda>0$ and every $i=1\dots, m$; in this case $L=\lambda K$.
\end{theorem}

\begin{proof}
By Theorem \ref{CharacterizationDualMixedVolumes1}, $\omega_i$, $i=0,\dots
m$, are the dual querma\ss integrals of two star bodies if and only if,
either there exist $0<a<b$ such that
\begin{equation}\label{e:cond_cone}
(\omega_0,\omega_1,\dots,\omega_m)\in\inter C_{a,b}^m,
\end{equation}
or $\omega_i=\lambda^i\omega_0$ for some $\lambda>0$, $i=1,\dots,m$. So,
we have to prove that \eqref{e:cond_cone} is equivalent to the fact that
the Hankel matrices $(a_{j,k})_{j,k=0}^r$, $(b_{j,k})_{j,k=0}^r$ are
positive definite.

First we assume \eqref{e:cond_cone} and consider the even case $m=2r$. On
one hand, for any $c_1,\dots,c_r\in\R$, the polynomial
\[
\left(\sum_{k=0}^rc_kt^k\right)^2=\sum_{j,k=0}^rc_jc_kt^{j+k}
\]
is always positive and so \eqref{t:polynom>0} yields
\[
\sum_{j,k=0}^rc_jc_k\omega_{j+k}>0.
\]
It shows that the Hankel matrix $(\omega_{j+k})_{j,k=0}^r$ is positive
definite. On the other hand, for any $d_1,\dots,d_r\in\R$, the polynomial
\[
(b-t)(t-a)\left(\sum_{k=1}^rd_kt^k\right)^2=\sum_{j,k=0}^rd_jd_k\bigl[(a+b)t^{j+k+1}-abt^{j+k}-t^{j+k+2}\bigr]
\]
is positive on $[a,b]$, and so \eqref{t:polynom>0} yields
\[
\sum_{j,k=0}^rd_jd_k\bigl[(a+b)\omega_{j+k+1}-ab\omega_{j+k}-\omega_{j+k+2}\bigr]>0.
\]
Therefore the matrix
$\bigl((a+b)\omega_{j+k+1}-ab\omega_{j+k}-\omega_{j+k+2}\bigr)_{j,k=0}^r$
is positive definite. One can argue the odd case $m=2r+1$ in a similar
way.

Conversely, we now assume that the Hankel matrices $(a_{j,k})_{j,k=0}^r$,
$(b_{j,k})_{j,k=0}^r$ are positive definite. Markov-Luk\'acs' theorem (see
\cite[Theorem~2.2 in c.~3]{KrNu}) provides a representation of the
non-negative polynomials on an interval $[a,b]$: any such polynomial
$P(t)$ of degree $m$ can be expressed as
\begin{equation}\label{e:P(t)>0}
P(t)=\left\{\begin{array}{ll}
\!\!\left(\sum_{i=0}^rc_it^i\right)^2\!+(b-t)(t-a)\left(\sum_{i=0}^{r-1}d_it^i\right)^2
    & \text{ if } m=2r,\\
\!\!(t-a)\left(\sum_{i=0}^rc_it^i\right)^2\!+(b-t)\left(\sum_{i=0}^{r-1}d_it^i\right)^2
    & \text{ if } m=2r\!+\!1
\end{array}\right.
\end{equation}
for $c_i,d_i\in\R$. Then, in the even case $m=2r$, for any positive
polynomial $P(t)$ written as in \eqref{e:P(t)>0} we have that
\[
\sum_{j,k=0}^rc_jc_k\omega_{j+k}+\sum_{j,k=0}^rd_jd_k\bigl[(a+b)\omega_{j+k+1}-ab\omega_{j+k}-\omega_{j+k+2}\bigr]>0
\]
because the matrices $(\omega_{j+k})_{j,k=0}^r$ and
$\bigl((a+b)\omega_{j+k+1}-ab\omega_{j+k}-\omega_{j+k+2}\bigr)_{j,k=0}^r$
are positive definite and so both summands are positive. When $m=2r+1$ we
argue in the same way. Thus, in both cases \eqref{t:polynom>0} shows
\eqref{e:cond_cone}, which concludes the proof.
\end{proof}

\subsection{New inequalities for dual querma\ss integrals and the moment problem}\label{ss:moment}

Next we prove Theorem \ref{t:family_ineq_dual}, which allows us to obtain
new inequalities for the dual querma\ss integrals.

\begin{proof}[Proof of Theorem \ref{t:family_ineq_dual}]
Let $f:\s^{n-1}\longrightarrow(0,\infty)$ be the function defined by
\[
f(u)=\frac{\rho_L(u)}{\rho_K(u)}
\]
and, for each sequence $(a_i)_{i\in\N}\subset\R$ and any $m\in\N$, let
$P_{1,m}$ and $P_{2,m}$ be the polynomials
\[
\begin{split}
P_{1,m}(x) & =\left(\sum_{i=0}^ma_ix^i\right)^2=\sum_{i,j=0}^ma_ia_jx^{i+j},\\
P_{2,m}(x) &
    =x\left(\sum_{i=0}^{m-1}a_ix^i\right)^2=\sum_{i,j=0}^ma_ia_jx^{i+j+1},
\end{split}
\]
which are positive for all $x\in(0,\infty)$. Consequently,
\[
\begin{split}
\frac{1}{n}\int_{\s^{n-1}}P_{1,m}\bigl(f(u)\bigr)\,\rho_K^n(u)\,\dlat\sigma(u) & >0\quad\text{ and}\\[2mm]
\frac{1}{n}\int_{\s^{n-1}}P_{2,m}\bigl(f(u)\bigr)\,\rho_K^n(u)\,\dlat\sigma(u)
 & >0,
\end{split}
\]
and thus, for every sequence $(a_i)_{i\in\N\cup\{0\}}$ and any $m\in\N$,
we have (cf. \eqref{e:dualW_i_int2})
\[
\begin{split}
\sum_{i,j=0}^ma_ia_j &
    \left(\frac{1}{n}\int_{\s^{n-1}}\!\!\!\left(\frac{\rho_L(u)}{\rho_K(u)}\right)^{i+j}\rho_K^n(u)\,\dlat\sigma(u)\right)
    =\sum_{i,j=0}^ma_ia_j\Wd_{i+j}(K,L)>0,\\
\sum_{i,j=0}^{m-1}a_ia_j &
    \left(\frac{1}{n}\int_{\s^{n-1}}\!\!\!\left(\frac{\rho_L(u)}{\rho_K(u)}\right)^{i+j+1}
    \!\!\!\!\rho_K^n(u)\,\dlat\sigma(u)\right)
    \!=\!\sum_{i,j=0}^{m-1}a_ia_j\Wd_{i+j+1}(K,L)>0.
\end{split}
\]
Therefore, $A_m$ and $B_m$ are positive definite matrices.
\end{proof}

We observe that Theorem \ref{t:family_ineq_dual} implies, in particular,
that for every $i<j$, $i,j\in\N\cup\{0\}$, the matrices
\[
\begin{pmatrix}
\Wd_{2i}(K,L) & \Wd_{i+j}(K,L)\\
\Wd_{i+j}(K,L) & \Wd_{2j}(K,L)
\end{pmatrix},\quad
\begin{pmatrix}
\Wd_{2i+1}(K,L) & \Wd_{i+j+1}(K,L)\\
\Wd_{i+j+1}(K,L) & \Wd_{2j+1}(K,L)
\end{pmatrix}
\]
are positive definite, and hence have a positive determinant. Thus we
obtain particular cases of the dual Aleksandrov-Fenchel inequalities
\eqref{e:special_dual_af}:
\[
\begin{split}
\Wd_{2i}(K,L)\Wd_{2j}(K,L) & >\Wd_{i+j}^2(K,L),\\
\Wd_{2i+1}(K,L)\Wd_{2j+1}(K,L) & >\Wd_{i+j+1}^2(K,L).
\end{split}
\]
Taking different submatrices we obtain a family of inequalities.

\begin{remark}
We would like to notice that the determinantal inequalities in
Corollary~\ref{c:ineq_dual} cannot be obtained from
\eqref{e:special_dual_af}. For instance, $\Delta_2>0$ does not hold if we
consider a sequence of numbers
$\omega_0>\omega_1>\omega_2=\omega_3=\omega_4$ (cf.~\eqref{e:Favard_dual})
satisfying also \eqref{e:special_dual_af}.
\end{remark}

We conclude this section by obtaining new inequalities for the dual
querma\ss integrals of two star bodies as a consequence of the moment
problem.

The moment problem seeks necessary and sufficient conditions for a
sequence $(m_i)_{i\in\N\cup\{0\}}$ to be the {\it moments} of some measure
$\mu$ on the real line. There are different versions of this problem, our
interest being focused to the case of a fixed interval $[a,b]$ (the
Hausdorff moment problem).

The solution to the Hausdorff moment problem is given by the following
result (see e.g. \cite[c.~3, Theorem~2.5]{KrNu}, cf.
Theorem~\ref{CharacterizationDualMixedVolumes2}).

\begin{theorem}\label{SolutionMomentProblem2}
Let $[a,b]\subset\R$ and let $(m_i)_{i\in\N\cup\{0\}}$ be a sequence of
positive numbers. Then, there exists a positive measure $\mu$ on $[a,b]$
with $m_i(\mu)=m_i$ for all $i\in\N\cup\{0\}$, if and only if, for every
$m\in\N$, the Hankel matrices $A_m=(a_{j,k})_{j,k=0}^m$ and
$B_m=(b_{j,k})_{j,k=0}^m$ are positive semi-definite, where either
\[
\begin{split}
\text{i) } & \; a_{j,k}=m_{j+k}\;\;\text{ and }\;\;
b_{j,k}=(a+b)m_{j+k+1}-ab\,m_{j+k}-m_{j+k+2},\;\text{ or}\\[1mm]
\text{ii) } & \; a_{j,k}=m_{j+k+1}-a\,m_{j+k}\;\;\text{ and }\;\;
b_{j,k}=b\,m_{j+k}-m_{j+k+1}.
\end{split}
\]
\end{theorem}

The above result provides new properties of dual querma\ss integrals. For
instance, we obtain the following result, whose proof is essentially
contained in the proofs of Theorems
\ref{CharacterizationDualMixedVolumes1} and
\ref{CharacterizationDualMixedVolumes2}.  It can be translated in further
inequalities between the querma\ss integrals.

\begin{proposition}\label{p:dual_quermass_Hankel}
Let $K,L\in\S^n_0$ be two star bodies and let
\[
a=\min_{u\in\s^{n-1}}\frac{\rho_L(u)}{\rho_K(u)}\quad\text{and}\quad
b=\max_{u\in\s^{n-1}}\frac{\rho_L(u)}{\rho_K(u)}.
\]
Then, for all $m\in\N$, the Hankel matrices $(a_{j,k})_{j,k=0}^m$,
$(b_{j,k})_{j,k=0}^m$ given by
\[
\begin{split}
a_{j,k} & =\Wd_{j+k+1}(K,L)-a\Wd_{j+k}(K,L),\\
b_{j,k} & =b\Wd_{j+k+1}(K,L)-\Wd_{j+k}(K,L),
\end{split}
\]
are positive semi-definite.
\end{proposition}

\begin{proof}
Following the notation in the proof of Theorem
\ref{CharacterizationDualMixedVolumes1}, we denote by
$f:\s^{n-1}\longrightarrow(0,\infty)$ the continuous function
$f(u)=\rho_L(u)/\rho_K(u)$, by $\nu$ the measure on the sphere
$\dlat\nu=(1/n)\rho_K^n(u)\,\dlat\sigma(u)$ and by $\mu$ the push-forward
measure of $\nu$ by $f$. Then, for every $i\in\N\cup\{0\}$ we have
\[
m_i(\mu)=\int_{\s^{n-1}}f(u)^i\,\dlat\nu(u)
=\frac{1}{n}\int_{\s^{n-1}}\left(\frac{\rho_L(u)}{\rho_K(u)}\right)^i\rho_K^{n}(u)\,\dlat\sigma(u)
=\Wd_i(K,L);
\]
by Theorem \ref{SolutionMomentProblem2} ii), the given Hankel matrices are
positive semi-definite.
\end{proof}

An analogous result can be obtained taking the Hankel matrices given by i)
in Theorem \ref{SolutionMomentProblem2}.

\section{The set of roots of dual Steiner polynomials}\label{s:roots}

We start collecting some properties on the behavior of the roots of dual
Steiner polynomials when the involved bodies slightly change. They will be
used for the proof of Theorem \ref{t:cone_half-open_monotone}. We notice
that these properties are analogous to the ones of the relative Steiner
polynomial; we include the proof for completeness.

\begin{lemma}\label{lem:roots_change}
Let $\gamma$ be a root of the dual Steiner polynomial $\f{K}{L}(z)$.
\begin{enumerate}
\item Let $\lambda>0$. Then $\lambda\,\gamma$ is a root of
$\f{\lambda\,K}{L}(z)$.
\item Let $\mu\geq 0$. Then  $\gamma-\mu$ is a root of $\f{K\widetilde{+}\mu L}{L}(z)$.
\item Let $\gamma=a+b\,\im$ with $a<0$, and let $0<\rho\leq 1$. Then
$a+(\rho\,b)\,\im$ is a root of $\f{\rho K\widetilde{+}(\rho-1)aL}{L}(z)$.
\end{enumerate}
\end{lemma}

\begin{proof}
Since $\Wd_i(\lambda\,K,L)=\lambda^{n-i}\,\Wd_i(K,L)$ for any
$i=0,\dots,n$, we have $\f{\lambda K}{L}(z)
=\lambda^n\,\f{K}{L}(z/\lambda)$, which shows i).

Now, for any non-negative numbers $\lambda,\mu\geq 0$, the radial addition
of star bodies satisfies $\mu L\widetilde{+}\lambda L=(\mu+\lambda)L$ (see
e.g. \cite[(2.2)]{Lut2}), and hence
\[
\Bigl|\bigl(K\widetilde{+}\mu L\bigr)\widetilde{+}\lambda L\Bigr|
=\bigl|K\widetilde{+}(\mu+\lambda)L\bigr|=\sum_{i=0}^n\binom{n}{i}\Wd_i(K,L)\,(\mu+\lambda)^i.
\]
Therefore $\f{K\widetilde{+}\mu L}{L}(z)=\f{K}{L}(\mu+z)$, which implies
ii).

Finally, iii) is just a combination of ii) and i), because the radial
addition also satisfies
$\rho\,K\widetilde{+}(\rho-1)aL=\rho\bigl(K\widetilde{+}[(\rho-1)a/\rho]L\bigr)$
(see again \cite[(2.2)]{Lut2}).
\end{proof}

The following proposition states that both the derivative and the
antiderivative of a dual Steiner polynomial are also dual Steiner
polynomials. This result will be also crucial in the proof of Theorem
\ref{t:cone_half-open_monotone}.

\begin{proposition}\label{p:derivative_integral}
Let $K,L\in\S^n_0$ be two star bodies.
\begin{enumerate}
\item[i)] There exist $K',L'\in\S^{n-1}_0$ such that
\[
\dfrac{\dlat\f{K}{L}}{\dlat z}(z)=\f{K'}{L'}(z).
\]
\item[ii)] There exist $K'',L''\in\S^{n+1}_0$ such that
\[
\dfrac{\dlat\f{K''}{L''}}{\dlat z}(z)=\f{K}{L}(z).
\]
\end{enumerate}
\end{proposition}

\begin{proof}
We first notice that since $\Wd_{n-i}(K,L)=\Wd_i(L,K)$ for all
$i=0,\dots,n$, Theorem~\ref{CharacterizationDualMixedVolumes1} ensures
that,
\begin{enumerate}
\item[(a)] either there exist $0<a<b$ such that
\[
\bigl(\Wd_n(K,L),\Wd_{n-1}(K,L),\dots,\Wd_0(K,L)\bigr)\in\inter C_{a,b}^n,
\]
\item[(b)] or $\Wd_{n-i}(K,L)=\lambda^i\Wd_n(K,L)$ for some $\lambda>0$ and
all $i=1,\dots,n$.
\end{enumerate}

First we prove i). We observe that the derivative of $\f{K}{L}(z)$ can be
expressed as
\[
\dfrac{\dlat\f{K}{L}}{\dlat z}(z)
=\sum_{i=1}^n\binom{n}{i}i\Wd_i(K,L)z^{i-1}=\sum_{i=0}^{n-1}\binom{n-1}{i}n\Wd_{i+1}(K,L)z^i.
\]
If (a) holds then, in particular,
\[
\bigl(\Wd_n(K,L),\Wd_{n-1}(K,L),\dots,\Wd_1(K,L)\bigr)\in\inter
C_{a,b}^{n-1},
\]
and thus also
$\bigl(n\Wd_n(K,L),n\Wd_{n-1}(K,L),\dots,n\Wd_1(K,L)\bigr)\in\inter
C_{a,b}^{n-1}$. So, there exist star bodies $L',K'\in\S^{n-1}_0$ such that
\[
n\Wd_{n-i}(K,L)=\Wdk{n-1}{i}(L',K')=\Wdk{n-1}{n-i-1}(K',L')
\]
for all $i=0,\dots,n-1$, where $\Wdk{j}{i}$ denotes the $i$-th dual
querma\ss integral in $\R^j$. Therefore $\bigl(\dlat\f{K}{L}/\dlat
z\bigr)(z)$ is the dual Steiner polynomial in $\R^{n-1}$ of the sets
$K',L'$.

If (b) holds, then $K=\lambda L$ (see Lemma~\ref{l:quermass_prop} ii)),
and hence
\[
\dfrac{\dlat\f{K}{L}}{\dlat z}(z)=n|L|(\lambda+z)^{n-1}.
\]
Then, taking $L'\in\S^{n-1}_0$ such that $|L'|_{n-1}=n|L|$ and $K'=\lambda
L'$, we obtain
\[
\f{K'}{L'}(z)=\sum_{i=0}^{n-1}\binom{n-1}{i}\Wdk{n-1}{i}(K',L')z^i=|L'|_{n-1}(\lambda+z)^{n-1}
=\dfrac{\dlat\f{K}{L}}{\dlat z}(z).
\]
It concludes the proof of i).

Now we prove ii). First we notice that the antiderivative of $\f{K}{L}(z)$
can be expressed, for any constant $C\geq 0$, as
\[
\begin{split}
\int\f{K}{L}(z)\,\dlat z & =C+\sum_{i=0}^n\binom{n}{i}\frac{1}{i+1}\Wd_i(K,L)z^{i+1}\\
 & =C+\sum_{i=1}^{n+1}\binom{n+1}{i}\frac{1}{n+1}\Wd_{i-1}(K,L)z^i.
\end{split}
\]
Again we start assuming that (a) holds. Let
$H=\span\{\e_1,\dots,\e_{n+1}\}\subset\R^{n+2}$ be the linear subspace
spanned by the first $n+1$ canonical vectors $\e_i$. Since
$C_{a,b}^n=C_{a,b}^{n+1}|H$ is the orthogonal projection of the cone
$C_{a,b}^{n+1}$ onto $H$, there exists a point
$(\omega_0,\dots,\omega_{n+1})\in\inter C_{a,b}^{n+1}$ such that
\[
(\omega_0,\dots,\omega_{n+1})|H=\frac{1}{n+1}\bigl(\Wd_n(K,L),\Wd_{n-1}(K,L),\dots,\Wd_0(K,L)\bigr).
\]
By Theorem \ref{CharacterizationDualMixedVolumes1}, there exist star
bodies $L'',K''\in\S^{n+1}_0$ such that, for all $i=0,\dots,n+1$,
$\omega_i=\Wdk{n+1}{i}(L'',K'')$ and, consequently,
\[
\frac{1}{n+1}\Wd_{n-i}(K,L)=\omega_i=\Wdk{n+1}{i}(L'',K'')=\Wdk{n+1}{n+1-i}\,(K'',L''),
\]
for all $i=0,\dots,n$. Then, setting $C=\omega_{n+1}$ we have
\[
\int\f{K}{L}(z)\,\dlat z
=\sum_{i=0}^{n+1}\binom{n+1}{i}\Wdk{n+1}{i}(K'',L'')z^i=\f{K''}{L''}(z).
\]
Thus, $\bigl(\dlat\f{K''}{L''}/\dlat z\bigr)(z)=\f{K}{L}(z)$, as required.

Finally, if (b) holds, then $K=\lambda L$ (see Lemma~\ref{l:quermass_prop}
ii)). So it suffices to take $L''\in\S^{n+1}_0$ such that
$|L''|_{n+1}=|L|/(n+1)$ and $K''=\lambda L''$, which yields
\[
\begin{split}
\dfrac{\dlat\f{K''}{L''}}{\dlat z}(z) & =\dfrac{\dlat}{\dlat z}
    \left[\sum_{i=0}^{n+1}\binom{n+1}{i}\Wdk{n+1}{i}(K'',L'')z^i\right]
    \!=\dfrac{\dlat}{\dlat z}\Bigl[|L''|_{n+1}(\lambda+z)^{n+1}\Bigr]\\
 & =(n+1)|L''|_{n+1}(\lambda+z)^n=|L|(\lambda+z)^n=\f{K}{L}(z).
\end{split}
\]
This concludes the proof of the proposition.
\end{proof}

We are now ready for the proof of Theorem \ref{t:cone_half-open_monotone}.
We just introduce an additional notation: for complex numbers
$z_1,\dots,z_r\in\C$ let
\begin{equation*}
\sy{i}{z_1,\dots,z_r}=\sum_{\substack{J\subseteq\{1,\dots,r\}\\\#J=i}}
\prod_{j\in J}z_j
\end{equation*}
denote the $i$-th elementary symmetric function of $z_1,\dots,z_r$, $1\leq
i\leq r$, setting $\sy{0}{z_1,\dots,z_r}=1$.

\begin{proof}[Proof of Theorem \ref{t:cone_half-open_monotone}]
First we prove item (a).

Clearly $\f{L}{L}(z)=|L|(z+1)^n$ because $\Wd_i(L,L)=|L|$ for all
$i=0,\dots,n$, and hence $-1$ is a root of the polynomial $\f{L}{L}(z)$.
Thus, by Lemma \ref{lem:roots_change} i), every $-c\in\R_{<0}$, $c>0$,
will be a root of $\f{cL}{L}(z)$ for any $L\in\S^n_0$, and so the negative
real axis $\R_{<0}$ is contained in $\rd(n)$.

Lemma \ref{lem:roots_change} i) also shows that $\rd(n)$ is a cone without
apex, i.e., if $\gamma\in\rd(n)$ and $\lambda>0$, then
$\lambda\,\gamma\in\rd(n)$. It remains to show that $\rd(n)$ is convex.
For the proof, let $\gamma_i=a_i+b_i\im\in\rd(n)$, $i=1,2$, and let
$\rho\in(0,1)$.

If both roots $\gamma_1,\gamma_2\in\R_{<0}$, i.e., $b_i=0$ for $i=1,2$,
then Lemma~\ref{lem:roots_change}~i) ensures that
$\rho\gamma_1+(1-\rho)\gamma_2$ is a root of $\f{cL}{L}(z)$ for
$c=\rho|a_1|+(1-\rho)|a_2|$ and any $L\in\S^n_0$. So we assume that at
least one of them has strictly positive imaginary part.

First we show that there exist $K_i,L\in\S^n_0$, $i=1,2$, such that
$\gamma_i$ is a root of $\f{K_i}{L}(z)$, $i=1,2$. We may assume, without
loss of generality, that $b_1\leq b_2$, which yields $b_2>0$. We suppose
moreover that $a_1\leq a_2\leq 0$. Let $K_2,L\in\S^n_0$ be such that
$\f{K_2}{L}(\gamma_2)=0$. By Lemma \ref{lem:roots_change} i), setting
$\lambda=b_1/b_2$, we have that $\lambda\gamma_2=\lambda a_2+b_1\im$ is a
root of $\f{\lambda K_2}{L}(z)$ (see Figure \ref{f:roots}).

\begin{figure}[h]
\begin{center}
\includegraphics[width=6.5cm]{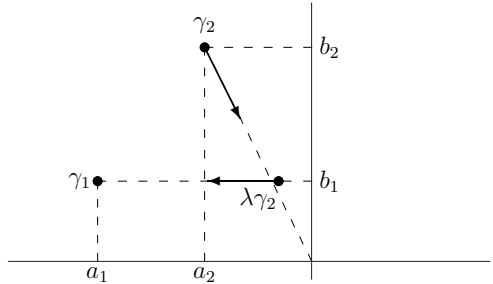}
\caption{Constructing $K_1\in\S^n_0$ such that
$\f{K_1}{L}(\gamma_1)=0$.}\label{f:roots}
\end{center}
\end{figure}

Furthermore, since $\lambda\leq 1$, we clearly have that $a_1\leq\lambda
a_2\leq 0$, and taking $\mu=\lambda a_2-a_1\geq 0$, Lemma
\ref{lem:roots_change} ii) ensures that $\gamma_1=\lambda\gamma_2-\mu$ is
a root of $\f{\lambda K_2\widetilde{+}\mu L}{L}(z)$. Thus, for
$K_1=\lambda K_2\widetilde{+}\mu L$ we get the desired property. The other
possibilities for $a_1,a_2$ can be argued in the same way by just choosing
properly the path between the two roots.

Finally, in order to prove the convexity of the cone, we just have to
construct a star body $M\in\S^n_0$ such that
$\rho\gamma_1+(1-\rho)\gamma_2$ is a root of $\f{M}{L}(z)$. Since at least
one of the roots has strictly positive imaginary part, we can take
$a_1/b_1=\max\{a_i/b_i:b_i>0,\,i=1,2\}$ and let
$\mu=b_1/\bigl(\rho\,b_1+(1-\rho)\,b_2\bigr)$. Then
\[
\nu=\mu\bigl(\rho\,a_1+(1-\rho)\,a_2\bigr)=b_1\frac{\rho\,a_1+(1-\rho)\,a_2}{\rho\,b_1+(1-\rho)\,b_2}\leq
a_1,
\]
because the above function is increasing in $\rho\in(0,1)$ by the choice
of $a_1/b_1$, and hence, Lemma \ref{lem:roots_change} ii) ensures that
$\nu+b_1\im$ is a root of $\f{K_1+(a_1-\nu)L}{L}$. Finally,
Lemma~\ref{lem:roots_change}~i) shows that $\rho\gamma_1+(1-\rho)\gamma_2$
is a root of the dual Steiner polynomial $\f{M}{L}(z)$ for
$M=(1/\mu)\,\bigl(K_1\widetilde{+}(a_1-\nu)L\bigr)$.

\smallskip

Next we prove item (b).

In order to show that $\rd(n)$ is half-opened, we are going to prove that
$C=\C^+\backslash\rd(n)$ is closed. Let $(z_i)_{i\in\N}\subset C$ be a
convergent sequence and let $z=\lim_{i\to\infty}z_i$. Clearly,
$z\not\in\R_{<0}$, and since $z_i\not\in\rd(n)$, then $z_i$ cannot be a
root of any dual Steiner polynomial.

For the sake of brevity we denote by $\Gamma$ any ($n-2$)-tuple of complex
numbers of the form
\[
\Gamma=\left\{\begin{array}{ll}
\bigl(\gamma_2,\overline{\gamma}_2,\dots,\gamma_{n/2},\overline{\gamma}_{n/2}\bigr)\in\C^{n-2}
    & \quad \text{if $n$ is even},\\[1mm]
\bigl(\gamma_2,\overline{\gamma}_2,\dots,\gamma_{(n-1)/2},\overline{\gamma}_{(n-1)/2},c\bigr)\in\C^{n-3}\times\R_{<0}
    & \quad\text{if $n$ is odd}.
\end{array}\right.
\]
We also denote by $\sigma:\C^n\longrightarrow\C^n\times\{1\}$ the
continuous map given by
\[
\sigma=\left((-1)^n\sigma_n,(-1)^{n-1}\dfrac{\sigma_{n-1}}{n},\dots,(-1)^i\dfrac{\sigma_i}{\binom{n}{i}},\dots,-\dfrac{\sigma_1}{n},1\right)
\]
Given an ($n-2$)-tuple $\Gamma$, let
$\sigma(z_i,\overline{z}_i,\Gamma)=(\omega_0^i,\dots,\omega_{n-1}^i,1)$
for all $i\in\N$. On one hand, since $z_i$ is not a root of any dual
Steiner polynomial, $\omega_j^i$, $j=0,\dots,n$, (we set $\omega_n=1$)
cannot be dual querma\ss integrals of any pair of star bodies, and hence,
by Theorem \ref{CharacterizationDualMixedVolumes1},
$(\omega_0^i,\dots,\omega_{n-1}^i,1)\not\in\inter C_{a,b}^n$ for any
$0<a<b$. Moreover, there exists no $\lambda>0$ such that
$\omega_j^i=\lambda^j\omega_0^i$.

On the other hand, since $\sigma$ is continuous, then
$(\omega_0^i,\dots,\omega_{n-1}^i,1\bigr)_{i\in\N}$ is a convergent
sequence and
\[
\lim_{i\to\infty}(\omega_0^i,\dots,\omega_{n-1}^i,1)
=:(\omega_0,\dots,\omega_{n-1},1)=\sigma(z,\overline{z},\Gamma)\not\in\inter
C_{a,b}^n
\]
for any $0<a<b$. Moreover, if there exists $\lambda>0$ such that
$\omega_j=\lambda^j\omega_0$ for all $j=0,\dots,n$, then
$\omega_j=\Wd_j(\lambda K,K)$ for some $K\in\S^n_0$, and thus $z$ would be
a root of the dual Steiner polynomial $\f{\lambda K}{K}(z)$, i.e.,
$z\in\R_{<0}$, which is not possible. Since this holds for any
($n-2$)-tuple $\Gamma$, we can conclude that $z\not\in\rd(n)$. It shows
that $C$ is closed and concludes the proof of (b).

\smallskip

Finally we show (c).

If $\gamma\in\rd(n)$, there exists a dual Steiner polynomial $\f{K}{L}(z)$
for $K,L\in\S^n_0$, such that $\f{K}{L}(\gamma)=0$. Then, by Proposition
\ref{p:derivative_integral} ii), we know there are star bodies
$K'',L''\in\S^{n+1}_0$ satisfying
\[
\dfrac{\dlat\f{K''}{L''}}{\dlat z}(z)=\f{K}{L}(z).
\]

Let $\gamma_1,\dots,\gamma_{n+1}$ be the roots of $\f{K''}{L''}(z)$.
Lucas' theorem (see e.g. \cite[Theorem~(6,1)]{Mard}) states that the roots
of the derivative of a polynomial lie in the convex hull of the roots of
the polynomial, and thus we get that
\[
\gamma\in\conv\{\gamma_1,\dots,\gamma_{n+1}\}\subset\rd(n+1),
\]
because $\rd(n+1)$ is convex.
\end{proof}

The different behavior of the roots of dual Steiner polynomials with
respect to the Steiner polynomial shows up also in the stability. We
recall that real polynomials whose zeros all have negative real part are
called {\em stable} or {\em Hurwitz}. In the next proposition we show
that, contrary to the classical case (cf. \cite[Proposition~1.3]{HHCS}),
$\rd(n)\subset\bigl\{z\in \C^+:\Real(z)<0\bigr\}$ if and only if $n=2$.

\begin{proposition}\label{p:r(2)_r(3)}
$\rd(2)=\left\{z\in\C^+:\Real(z)<0\right\}$. Moreover, a negative real
number is root of $\f{K}{L}(z)$ for $K,L\in\S^2_0$ if and only if $K,L$
are dilates. For $n\geq 3$ there exist non-stable dual Steiner
polynomials.
\end{proposition}

\begin{proof}
The roots of $\f{K}{L}(z)=\Wd_0(K,L)+2\Wd_1(K,L)z+\Wd_2(K,L)z^2$, for
$K,L\in\S^2_0$, are
\[
\gamma_1,\gamma_2=\dfrac{-\Wd_1(K,L)\pm\sqrt{\Wd_1(K,L)^2-\Wd_0(K,L)\Wd_2(K,L)}}{\Wd_2(K,L)},
\]
and the dual Aleksandrov-Fenchel inequality \eqref{e:special_dual_af} for
$i=0,j=1,k=2$ shows that they cannot be real numbers unless $K$ and $L$
are dilates.

Next we observe that, given a triple $(\omega_0,\omega_1,\omega_2)$ of
positive numbers, the dual Aleksandrov-Fenchel inequality
$\omega_1^2<\omega_0\omega_2$ is equivalent to the fact that the pair
\[
\left(\frac{\omega_1}{\omega_0},\frac{\omega_2}{\omega_0}\right)\in\inter\conv\bigl\{(t,t^2):t>0\bigr\},
\]
which implies that $\omega_i$, $i=0,1,2$, are the dual querma\ss integrals
of two planar star bodies $K,L\in\S^2_0$.

Now, let $a+b\,\im\in\bigl\{z\in\C^+:\Real(z)<0\bigr\}$, $b>0$. Then the
positive numbers
\[
\omega_0=a^2+b^2,\quad \omega_1=-a\quad\text{ and }\quad \omega_2=1
\]
determine a polynomial $\omega_0+2\omega_1z+\omega_2z^2$ having $a+b\,\im$
as a root, and clearly satisfy the inequality
$\omega_1^2<\omega_0\omega_2$, which implies that they are dual querma\ss
integrals of two planar star bodies. Therefore $a+b\,\im\in\rd(2)$.

Finally, for $n=3$, the Li\'enard-Chipart criterion for stability of
polynomials (see e.g. \cite[Theorem~(40,3)]{Mard}) allows to check that
there are dual Steiner polynomials having roots with positive real part.
\end{proof}

Proposition \ref{p:r(2)_r(3)} implies, in particular, that the inclusion
$\rd(2)\subset\rd(3)$ between the lowest dimensional cones is strict.

\begin{remark}
For $n=3$, if $a+b\,\im\in\C^+$, $a,b>0$, is a root of a dual Steiner
polynomial $\f{K}{L}(z)$ for some $K,L\in\S^3_0$, and $-c$, $c\geq 0$, is
the real root, we immediately have the identities
\[
c-2a=3\frac{\Wd_2(K,L)}{\Wd_3(K,L)},\quad
a^2+b^2-2ac=3\frac{\Wd_1(K,L)}{\Wd_3(K,L)},\quad
c(a^2+b^2)=\frac{\Wd_0(K,L)}{\Wd_3(K,L)}.
\]
Then, inequalities \eqref{e:special_dual_af} allow to see that
$b>\sqrt{3}a$. Therefore, the cone
$\rd(3)\subseteq\left\{a+b\,\im\in\C^+:b>\sqrt{3}a\right\}$.
\end{remark}

We have seen that two different negative real numbers cannot be the roots
of a $2$-dimensional dual Steiner polynomial. The same occurs in arbitrary
dimension, which states another difference with the classical Steiner
polynomial, where this situation is possible (see
\cite[Proposition~2.3]{HHCS}).

\begin{proposition}
For any $K,L\in\S^n_0$, all roots of $\f{K}{L}(z)$ are real if and only if
they are all equal.
\end{proposition}

\begin{proof}
We suppose there exist $K,L\in\S^n_0$ and
$\gamma_1,\dots,\gamma_n\in\R_{<0}$ such that $\f{K}{L}(\gamma_i)=0$ for
all $i=1,\dots,n$. Then, Newton inequalities (see e.g. \cite{HLP}) ensure
that the elementary symmetric functions of the $\gamma_i$'s satisfy
\[
\left(\frac{\sy{j}{\gamma_1,\dots,\gamma_n}}{\binom{n}{j}}\right)^2\geq
\frac{\sy{j-1}{\gamma_1,\dots,\gamma_n}}{\binom{n}{j-1}}\frac{\sy{j+1}{\gamma_1,\dots,\gamma_n}}{\binom{n}{j+1}}.
\]
Since
\[
\sy{j}{\gamma_1,\dots,\gamma_n}=(-1)^j\binom{n}{j}\dfrac{\Wd_{n-j}(K,L)}{\Wd_n(K,L)},
\]
the above inequality translates into
\[
\Wd_{n-j}(K,L)^2\geq\Wd_{n-j+1}(K,L)\Wd_{n-j-1}(K,L)
\]
which, together with the dual Aleksandrov-Fenchel inequalities
\eqref{e:special_dual_af} for $i=j-1$ and $k=j+1$, yields
$\Wd_{n-j}(K,L)^2=\Wd_{n-j+1}(K,L)\Wd_{n-j-1}(K,L)$. Then $K=\lambda L$
for some $\lambda>0$, and hence $\gamma_i=-\lambda$ for all $i=1,\dots,n$,
a contradiction.
\end{proof}

\end{document}